%% file: power-decoding-ag.tex
\documentclass[conference,a4paper]{IEEEtran}

\input{defs}

\title{Improved Power Decoding of \\ Algebraic Geometry Codes}

\author{\IEEEauthorblockN{Sven Puchinger, Johan Rosenkilde, Grigory Solomatov}
\IEEEauthorblockA{
Department of Applied Mathematics and Computer Science, \\
Technical University of Denmark (DTU), Denmark
}}
\IEEEoverridecommandlockouts

\begin{document}
\maketitle

\begin{abstract}
  Power decoding is a partial decoding paradigm for arbitrary algebraic geometry codes for decoding beyond half the minimum distance, which usually returns the unique closest codeword, but in rare cases fails to return anything.
  The original version decodes roughly up to the Sudan radius, while an improved version decodes up to the Johnson radius, but has so far been described only for Reed--Solomon and one-point Hermitian codes.
  In this paper we show how the improved version can be applied to any algebraic geometry code.
\end{abstract}


\section{Introduction}

\subsection{Related work}

Power decoding was originally proposed in \cite{schmidt_decoding_2006} and \cite{schmidt_syndrome_2010} as a method for decoding Reed--Solomon codes beyond half the minimum distance, achieving roughly the same decoding radius as Sudan's list decoder \cite{sudan_decoding_1997}.
It is a partial unique decoder, meaning that it either returns a unique decoding solution, or declares failure.
It has been demonstrated through numerical simulations that power decoding has a very low probability of failure for random errors;
there are proofs of this for small parameter choices \cite{schmidt_syndrome_2010,nielsen_power_2014}, but the general case remains to be an open problem.

In \cite{rosenkilde_power_2018} it was shown how to incorporate a multiplicity parameter into power decoding, similar to that of Guruswami--Sudan (GS) algorithm \cite{guruswami1998improved}, thereby matching the latter's decoding radius, called the \emph{Johnson radius}.
Both the original and the improved power decoding with multiplicity has been applied to one-point Hermitian codes, which belong to the large family of \emph{algebraic geometry} (AG) \emph{codes} that were introduced by Goppa in 1981 \cite{goppa1981algebraic, goppa1983algebraico}.
They are a natural generalization of Reed--Solomon codes and some of their subfamilies have been shown to have minimal distance that beats the Gilbert--Varshamov bound \cite{tsfasman_modular_1982}.

At the heart of power decoding lies a set of ``key equations'', which can be solved by linear algebra, or by faster, more structured approaches, e.g. by generalizing the Berlekamp--Massey or Euclidean algorithm, see e.g.~\cite{sidorenko_linear_2011,rosenkilde_algorithms_2019} and their references.
This step is algebraically connected with to the ``interpolation step'' of the GS algorithm, which may be computed in similar complexity \cite{chowdhury_faster_2015,rosenkilde_algorithms_2019}.

After interpolation, the GS algorithm proceeds with a ``root-finding step'', which is not required by power decoding. For Reed--Solomon and one-point Hermitian codes, root-finding is asymptotically faster than the interpolation step\footnote{%
  Except if the code length is very short compared to the field size: a sub-routine of the root-finding step is to find $\field_q$-roots in $\field_q[x]$ polynomials.
} \cite{neiger_fast_2017,chowdhury_faster_2015,nielsen_sub-quadratic_2015}, but in hardware-implementations root-finding can take up significant circuit area \cite{ahmed_vlsi_2011}.
For more general AG codes, the root-finding step is not widely studied, and it is unclear whether it is faster than interpolation; e.g. \cite{wu_efficient_2001} describes root-finding for one-point codes which is not superior to performing interpolation using linear algebra.
Further, both power decoding and the GS algorithm generalize to interleaved codes, and here the root-finding step is currently the bottleneck in the GS algorithm \cite{cohn2013approximate}, and so power decoding is asymptotically faster \cite{puchinger2017decoding}.

Recently, \cite{couvreur2020power} generalized the original power decoding to apply to general AG codes.
Their work was principally motivated by their new variant of (the original) power decoding, called power error-locating pairs, which may be applied to attack AG-code-based cryptosystems in a setting where it seems the GS decoder can not be used.

\subsection{Contributions}
In this paper we:
\begin{itemize}
\item formulate improved power decoding \cite{rosenkilde_power_2018} in the language of function fields and adapt it to general AG codes,
\item show how the resulting \emph{key equations} can be solved using linear algebra, and
\item derive the decoding radius for the proposed method under mild assumptions on the code and verify it using numerical simulations. 
\end{itemize}
The most important takeaway is that our decoding radius coincides with the one from \cite{rosenkilde2018power} for Reed--Solomon codes and the one from \cite{puchinger2019improved} for one-point Hermitian codes, achieving the Johnson radius for suitable decoding parameters. Though the computational complexity of our decoder can likely be improved by replacing linear algebra with more advanced techniques, this remains outside of the scope of this paper.

Our function field modelling revolves around ``using an extra place'' $P_\infty$ as in Pellikaan \cite{porter1992decoding}, which allows controlling the ``size'' of functions by their pole order at $P_\infty$.
Our key equations are in the style of Gao \cite{gao_new_2003}, as introduced for power decoding in \cite{nielsen_power_2014}.

We present evidence supporting that our decoder achieves a decoding radius similar to the GS algorithm in two ways: we theoretically derive a radius at which the decoder will surely fail, and then we see in simulations that for fewer random errors, the decoder seems to almost always succeed.
We have no theoretic bound on this failure probability, which has so far proved very challenging even in the case of Reed--Solomon codes, see e.g.~\cite{rosenkilde2018power}.
On the other hand, advantages of the power decoding compared to GS algorithm include:
\begin{itemize}
  \item The power decoder is structurally simpler than the GS algorithm since it does not have a root-finding step, making it easier to implement, possibly faster, and with less footprint in hardware.
\item It is reasonable that a generalization to interleaved AG codes is possible as in \cite{puchinger2017decoding,puchinger2019improved}.
\item Potentially apply power decoding in a more efficient message attack on AG-based cryptosystems than in \cite{couvreur2020power}, possibly by first reformulating it in terms of power error-locating pairs.
\end{itemize}

\section{Preliminaries}
In this section we briefly introduce AG codes as \emph{evaluation codes} and formulate the decoding problem that we wish to solve. We assume that the reader is familiar with the language of function fields, which is presented in great detail in \cite{stichtenoth_algebraic_2009}.

Let $\ff$ be a function field having genus $g$ over a base field $\field$.
Let $P_1,\dots,P_n, P_{\infty}$ be rational places of $\ffield$,
and fix two divisors $D = P_1 + \cdots + P_n$ and $G$ with $\deg G = \gamma$, such that $\supp G \cap \supp D = \emptyset$.
The code that we will be considering is
\[
  \code_{D,G} = \{
  \big(
  f(P_1),\dots,f(P_n)
  \big)
  \in \field^n
  \;|\;
  f \in \rrspace(G)
  \}  \ ,
\]
where
$\rrspace(G) = \{f \in \ffield \;|\; (f) \geq -G\} \cup \{0\}$
is the \emph{Riemann-Roch space} of $G$. The dimension of this code is $k = \rrdim(G)$, where $\rrdim(A)$ denotes the dimension of $\rrspace(A)$ for any divisor $A$; and the code's minimum distance is bounded from below by the \emph{designed minimum distance} $d^* = n - \degG$.

 The problem that we wish to address is as follows:
 \begin{problem}
   \label{prob:decoding}
  Let $f \in \rrspace(G)$ be a \emph{message} and let $\vec{c} \in \code_{D,G}$ be the corresponding \emph{codeword}. Given the \emph{received word}
  \[
    \vec{r} = (r_1,\dots,r_n) = \vec{c} + \vec{e} \ ,
  \]
  where $\vec{e} = (e_1,\dots,e_n) \in \field^n$ is some unknown error, recover $f$.
\end{problem}

In the next section, in the traditional spirit of power decoding, we show how \cref{prob:decoding} can be reformulated as a highly non-linear system of equations.

\section{The Key Equations}
\label{sec:key-eqs}
In this section we formulate the so called \emph{key equations} that lie at the heart of power decoding using the language of function fields.

We begin by defining the \emph{error locator} and the \emph{received word interpolator}, following up with some results about their sizes.

\begin{definition}
  If $\vec{r} = (r_1,\dots,r_n)$ is the received word, then an $\vec{r}$-\emph{interpolator} of degree $\degR$ is an element
  \begin{align*}
    R &\in \rrspace(G + \degR \Pinf) \quad \text{such that} \\
    R &\not \in \rrspace\big(
    G + (\degR-1) \Pinf
    \big) \ ,
  \end{align*}
\end{definition}
and $R(P_i) = r_i$ for $i = 1,\dots,n$.

\begin{lemma}
  There exists an $\r$-interpolator $R$ with degree $\degR$ satisfying $\degR \leq n - \degG + 2g - 1$.
\end{lemma}

\begin{proof}
  Consider the $\field$-linear map
  \begin{align*}
    \phi :
    \begin{cases}
      \rrspace(G + \degR \Pinf) \rightarrow \field^n \\
      h \mapsto \big(
      h(P_1),\dots,h(P_n)
      \big)
    \end{cases}    
  \end{align*}
  We are guaranteed that $R$ exists if the dimension of the image of $\phi$, which is given by the rank-nullity theorem as
  \[
    \dim \im \phi = \rrdim(G + \degR P_{\infty}) - \rrdim(G + \degR P_{\infty} - D) \ ,
  \]
  is $n$. Indeed, Riemann's theorem theorem says that
  \[
    \rrdim(G + \degR P_{\infty}) \geq \degG + \degR + 1 - g \ ,
  \]
  and if $\degR \geq n - \degG + 2g - 1$, then \cite[Thm 1.5.17]{stichtenoth_algebraic_2009} promises that
  \[
    \rrdim(G + \degR P_{\infty} - D) = \degG + \degR - n + 1 - g \ ,
  \]
  which shows that the image of $\phi$ has dimension $n$.
\end{proof}

\begin{definition}
  Let $E = \sum_{i \in \errors} P_i$, where $\errors = \{i \;|\; e_i \neq 0\}$ is the set of error positions. An \emph{error locator} with multiplicity $s$ and degree $\degLs$ is an element
  \begin{align*}
    \Ls &\in \rrspace(\degLs \Pinf - sE) \quad \text{such that} \\
    \Ls &\not \in \rrspace \big(
          (\degLs-1) \Pinf - sE
          \big) \ .
  \end{align*}
\end{definition}

\begin{lemma}
  \label{lem:degLs}
  For any $\degLs \geq s|\errors| + g$ there exists an error locator with degree at most $\degLs$, however, if $\degLs < s|\errors|$, then no such error locator exists.
\end{lemma}
\begin{proof}
  If $\degLs \geq s|\errors| + g$, then the Riemann-Roch theorem promises that
  \[
    \rrdim(\degLs P_{\infty} - sE)
    \geq \degLs - s|\errors| + 1 - g
    \geq 0 \ ,
  \]
  which guarantees existence of an error locator with degree at most $\degLs$.

  If on the other hand $\degLs < s|\errors|$, then $\Ls$ must have more zeroes than poles, which is impossible.
\end{proof}

It will be convenient to define the following divisors:

\begin{align*}
  V_t &= \degLs \Pinf + tG  \ , \\
  Q_t &= \degLs \Pinf + t(G + \degR \Pinf) \quad \text{and} \\
  W_j &= \degLs \Pinf + j(G + \degR \Pinf - D),
\end{align*}
for $j = 0,\dots,s-1$ and $t = 1,\dots,\ell$.

The next lemma relates these divisors to $f$, $\Ls$ and $R$.
\begin{lemma}
  \label{lem:LfR}
  If $f$ is the sent message, $R$ is an $\vec{r}$-interpolator and $\Ls$ is an error locator, then for $t = 1,\dots,\ell$ it holds that
  \[
    \Ls f^t \in \rrspace(V_t) \ ,
  \]
  and for $j = 0,\dots,s-1$ it holds that
  \[
    \Ls (f-R)^jR^{t-j} \in
    \begin{cases}
      \rrspace
      \big(
      W_j + (t-j)(G + \degR \Pinf)
      \big) &\text{ if } j < s\\
      \rrspace(Q_t - sD) &\text{ if } j \geq s 
    \end{cases} \ .
  \]
\end{lemma}

\begin{proof}
  The first claim is obvious. For the second claim observe that
  \begin{align*}
    f-R &\in \rrspace
            \big(
            G + \degR \Pinf - D + E
            \big) \ .
  \end{align*}
  If $j < s$, then
  \begin{align*}
    \Ls (f-R)^j &\in \rrspace
                  \big(\degLs \Pinf + j(G + \degR \Pinf) - sE - j(D-E)
                  \big) \\
                &\in \rrspace
                  \big(\degLs \Pinf + j(G + \degR \Pinf - D) - (s-j)E
                  \big) \\
                &\subset \rrspace(W_j) \ .
  \end{align*}
  If on the other hand $j \geq s$, then
  \begin{align*}
    \Ls (f-R)^jR^{t-j} &\in \rrspace
                  \big(
                  \degLs \Pinf + t(G + \degR \Pinf) - sE - j(D-E)
                  \big) \\
                &\subset \rrspace
                  \big(
                  \degLs \Pinf + t(G + \degR \Pinf) - sE - s(D-E)
                  \big) \\
                &= \rrspace(Q_t - sD) \ .
  \end{align*}
\end{proof}

We are now ready to present the key equations, which express an algebraic relationship between the sent message $f$, an $\r$-interpolator $R$ and a corresponding error locator $\Ls$.

\begin{theorem}[The Key Equations]
  \label{thm:key-eqs}
  For $t = 1,\dots,\ell$
  \begin{align}
    \begin{split}
      \label{eq:key}
      \Ls f^t - \sum_{j=0}^{\min\{t,s-1\}} \binom{t}{j} \Ls (f-R)^j R^{t-j} \\
      \in
      \begin{cases}
        \{0\} &\text{ if } 1 \leq t \leq s-1 \\
        \rrspace(Q_t - sD) &\text{ if } s \leq t \leq \ell
      \end{cases} \ .
    \end{split}
  \end{align}
\end{theorem}

\begin{proof}
  We have that
  \begin{align*}
    \Ls f^t
    &= \Ls
      \big(
      (f-R) + R
      \big)^t \\
    &= \sum_{j=0}^t \binom{t}{j} \Ls (f-R)^j R^{t-j} \ ,
  \end{align*}
  which proves the claim when $1 \leq t \leq s-1$, while if $s \leq t \leq \ell$, then it follows from \cref{lem:LfR} that
  \begin{align*}
    &\Ls f^t - \sum_{j=0}^{\min\{t,s-1\}} \binom{t}{j} \Ls (f-R)^j R^{t-j} \\
    &= \sum_{j=s}^t \binom{t}{j} \Ls (f-R)^j R^{t-j} \in \rrspace(Q_t - sD)
  \end{align*}
\end{proof}

Since \eqref{eq:key} is satisfied by the sent message, it makes sense to look for a solution of \eqref{eq:key} in order to address \cref{prob:decoding}. We explain how this can be done in the next section.

\section{Solving the Key Equations}\label{sec:lin-sys}

In this section we show how to solve the following problem, which is motivated by \cref{thm:key-eqs} and \cref{prob:decoding}:

\begin{problem}
\label{prob:linearized_problem}
  Consider an instance of \cref{prob:decoding}.
  Fix $\degLs\geq 0$ and let $R$ be an $\r$-interpolator.
  Find $\phi_t \in \rrspace(V_t)$ for $t=1,\dots, \ell$ and $\psi_j \in \rrspace(W_j)$ for $j=0,\dots,s-1$, not all zero, such that
\begin{align*}
& \phi_t - \sum_{j=0}^{\min\{t,s-1\}} \binom{t}{j} \psi_j R^{t-j}  \notag\\
&\quad \quad \in \begin{cases}
\{0\} &1 \leq t \leq s-1 \\
\rrspace(Q_t - sD) &s \leq t \leq \ell
\end{cases} \ . \\
\end{align*}
\end{problem}

According to \cref{thm:key-eqs}, if $f$ is the sent message and there exists an error locator $\Ls$ with degree at most $\degLs$, then among the possible solutions of \cref{prob:linearized_problem} we would find
\begin{align}
\phi_t &= \Ls f^t \notag \\
\psi_j &= \Ls(f-R)^j. \label{eq:special_solution}
\end{align}

Since $\deg \Ls$ is roughly proportional to the number of errors (see \cref{lem:degLs}), this means that a solution of \cref{prob:linearized_problem} of the type \eqref{eq:special_solution} for the smallest-possible $\degLs$ corresponds to a codeword that is among the closest codewords to our received word.\footnote{Note that it might happen that for two errors of similar weight, the error locator degree of the one of smaller weight is larger or equal to the one of larger weight. For large enough parameter $s$, this case becomes less likely for random errors.}
It might be that \cref{prob:linearized_problem} for small $\degLs$ has solutions that do not correspond to an error locator/message pair.
This might happen since \cref{prob:linearized_problem} poses weaker constraints on its solutions compared to \eqref{eq:key}. Previous works on (improved) power decoding indicate that such ``spurious'' solutions only occur with very small probability. Also our simulation results in \cref{sec:num_res} indicate this. We will get back to types of failure in the next section.

These arguments motivate our decoding strategy: Find a solution of \cref{prob:linearized_problem} for the smallest $\degLs$ for which the problem has a solution.
If we are successful, we get a solution of \cref{prob:linearized_problem} of the form in \eqref{eq:special_solution} (or a scalar multiple thereof), where $\Ls$ and $f$ correspond to a close codeword to $\r$. We can extract $f$, and thus $\c$ from this solution by division $\tfrac{\phi_1}{\psi_0} = \tfrac{\Ls f^1}{\Ls} = f$.

Since computational complexity is not a primary concern here, we describe in this section how to solve the problem for a fixed $\degLs$. The ``smallest solution'' can then be found by iterating $\degLs$ from $0$ to the maximal possible value for which we expect to be able to decode.

The entire decoding strategy is outlined in Algorithm~\ref{alg:decoder}.

\begin{algorithm}
\DontPrintSemicolon
\KwIn{Received word $\r \in \mathbb{F}^{n}$ and positive integers $s \leq \ell$}
\KwOut{$f \in \rrspace(G)$ such that $\c = (f(P_1),\dots,f(P_n))$ is a codeword with corresponding minimal error locator degree; or ``decoding failure''.}
\For{$\degLs =0,\dots,sn+g$}{
\If{\cref{prob:linearized_problem} with parameter $\degLs$ has a unique solution up to scalar multiples \label{line:check_if_it_has_solution}}{
$(\phi_t)_{t=1,\dots,\ell},(\psi_j)_{j=0,\dots,s-1} \gets$ a solution of \cref{prob:linearized_problem} \label{line:finding_solution} \\
$\hat{f} \gets \tfrac{\phi_1}{\psi_0}$ \\
\If{$\hat{f} \in \rrspace(G)$ and $\psi_0$ is a valid error locator for the error $\hat{\e} = \r-\hat{\c}$, where $\hat{c}$ is the codeword corresponding to $\hat{f}$}{ %
\Return{$\hat{f}$}
}\Else{\Return{``decoding failure''}}
}
}
\Return{``decoding failure''}
\caption{Decoder}
\label{alg:decoder}
\end{algorithm}

In the remainder of this section we give an explicit construction of a matrix $\mat{U}$ over $\field$ whose right kernel contains vectors from which solutions to \cref{prob:linearized_problem} can be obtained, i.e., how to implement Lines~\ref{line:check_if_it_has_solution} and \ref{line:finding_solution} of Algorithm~\ref{alg:decoder}.

 For any divisor $A$ let $\basis_A \in \rrspace(A)^{\rrdim(A)}$ denote a vector whose entries form \emph{any} $\field$-basis of $\rrspace(A)$ such that for any two divisors $A$ and $B$ it holds that
  \[
    \rrspace(A) = \rrspace(B) \iff \basis_A = \basis_B \ .
  \]
  
  For any $h \in \rrspace(A)$ let $h_A \in \field^{\rrdim(A)}$ be the unique vector such that $\basis_A \cdot h_A = h$, and for any $p \in \rrspace(B-A)$ let $p_{B,A} \in \field^{\rrdim(B) \times \rrdim(A)}$ be the matrix whose $i$-th column is $(p a_i)_B$, where $(a_1,\dots,a_{\rrdim(A)}) = \basis_A$. It is easy to see that
  \[
    \basis_B p_{B,A} h_A = p \basis_A \cdot h_A = ph \in \rrspace(B) \ .
  \]

In order to construct $\mat{U}$ we will need a satisfactory received word interpolator $R$ which can be for example be obtained as follows: If $\mat{M} \in \field^{n \times \rrdim(H)}$, where $H = \degR \Pinf + G$, is the matrix whose $i$-th column is
\[
    [b_i(P_1) \ \cdots \ b_i(P_n) ]^T \ ,
  \]
  where $(b_1,\dots,b_{\rrdim(H)}) = \basis_{H}$, then $R = \basis_H \cdot \vec{z}$, where $\vec{z}$ is any solution to the linear system $\mat{M} \vec{z} = \vec{r}$.

The next lemma together with the the subsequent \cref{cor:U} give an explicit construction of $\mat{U}$ and show its relation to \cref{prob:linearized_problem}.
  \begin{lemma}
    \label{lem:U_k}
  Let
  \[
    \vec{u} = (\vec{v}_1|\dots|\vec{v}_{\ell}|\vec{w}_0|\dots|\vec{w}_{s-1})
    \in \field^{\vars} \ ,
  \]
  where
    \begin{align*}
    \vec{v}_r &= (\phi_r)_{V_r} \ , \\
    \vec{w}_j &= (\psi_j)_{W_j} \quad \text{and} \\
    \vars &= \sum_r \rrdim(V_r) + \sum_j \rrdim(W_j)
  \end{align*}
   for $r = 1,\dots,\ell$ and $j = 0,\dots,s-1$, where the $\phi_r,\psi_j$ are a solution to \cref{prob:linearized_problem}. If $\mat{U}_t \in \field^{\rrdim(Q_t) \times \vars}$ is the matrix defined by
  \[
    \mat{U}_t =
    \begin{bmatrix}
      \mat{V}_{t,1} | \cdots | \mat{V}_{t,\ell} | \mat{W}_{t,0} | \cdots | \mat{W}_{t,s-1}  
    \end{bmatrix} \ ,  
  \]
  where
  \begin{align*}    
    \mat{V}_{t,r} &=
                    \begin{cases}
                      1_{Q_t,V_r}
                      &\text{ if } t = r \\
                      0_{Q_t,V_r}
                      &\text{ if } t \neq r \\
                    \end{cases} \quad \text{and}\\
    \mat{W}_{t,j} &=
                    \begin{cases}
                      -\binom{t}{j} (R^{t-j})_{Q_t,W_j}
                      &\text{ if } j \leq t \\
                      0_{Q_t,W_j} &\text{ if } j > t \\
                    \end{cases} \ ,
  \end{align*}
  for $t = 1,\dots,\ell$, then
  \[
    \basis_{Q_t} \mat{U}_t \vec{u} \in 
    \begin{cases}
      \{0\} &\text{ if } 1 \leq t \leq s-1 \\
      \rrspace(Q_t - sD) &\text{ if } s \leq t \leq \ell
    \end{cases} \ .
  \]
\end{lemma}

\begin{proof}
  Simply observe that
  \begin{align*}
    \basis_{Q_t} \mat{U}_t \vec{u}
    &= \basis_{Q_t} \sum_{r=1}^{\ell} \mat{V}_{t,r}\vec{v}_r + \basis_{Q_t} \sum_{j=0}^{s-1} \mat{W}_{t,j} \vec{w}_j \\
    &= \phi_t - \sum_{j=0}^{\min\{t,s-1\}} \binom{t}{j} \psi_j R^{t-j} \ .
  \end{align*}
\end{proof}

\begin{corollary}[of \cref{lem:U_k}]
  \label{cor:U}
  If
  \[
    \mat{U} = \left[
      \begin{array}{cc}
        \mat{K}_1 \mat{U}_1 \\
        \hline 
        \vdots \\
        \hline 
        \mat{K}_{\ell} \mat{U}_{\ell}
      \end{array}\right] \ ,
  \]
  where  
  \[
    \mat{K}_t = 
    \begin{cases}
      1_{Q_t,Q_t} &\text{ if } 1 \leq t \leq s-1 \\
      \text{left kernel matrix of }1_{Q_t,Q_t - sD} &\text{ if } s \leq t \leq \ell
    \end{cases} \ 
  \]
  and $\vec{u}$ is as in \cref{lem:U_k}, then 
  \begin{eqnarray}
    \label{eq:Uv=0}
    \mat{U} \vec{u} = \vec{0} \ .
  \end{eqnarray}
\end{corollary}

In the next section we describe the conditions under which decoding is expected to fail.

\section{Decoding Radius}
\label{sec:decoding-radius}

We turn to investigating the decoding performance of the proposed decoder.
Since the decoder returns at most one codeword while attempting to decode beyond half the minimum distance, it must obviously fail for certain received words.
Theoretically bounding this probability seems a difficult problem even for the simplest case of Reed--Solomon codes, and only partial results are known \cite{schmidt_decoding_2006,rosenkilde_power_2018}.
We follow the approach of several previous papers on power decoding: we derive theoretically a number of errors at which our decoder is always guaranteed to fail, and we \emph{define} this to be the ``decoding radius''; this name will be supported by simulations in \cref{sec:num_res} which indicates that the decoding algorithm succeeds with high probability whenever fewer, random errors occur.

It can be seen from Algorithm~\ref{alg:decoder} that our decoder declares a decoding failure if and only if there is no value of $\degLs$ for which the matrix $\mat U$ in \cref{cor:U} has a right kernel of $\field_q$-dimension exactly 1.

On the other hand, when we do \emph{not} declare decoding failure, the returned message candidate $\hat f \in \mathcal L(G)$ must correspond to the closest codeword to the received word, as in the discussion in the preceding section.

This notion of decoding failure motivates a natural bound for when we -- simply by linear algebraic arguments -- would know a priori that the decoder will fail in the generic case when the smallest error locator corresponding to a codeword has degree $\degLs = s |\mathcal{E}| + g$ (and not smaller), where $|\mathcal{E}|$ is the distance of the codeword to the received word:
\begin{definition}
\label{def:decoding_radius}
For a given code $\mathcal C$, the \emph{decoding radius} of power decoding, denoted $\taumax$, is the greatest value of $\tau$ such that for $\degLs=s \tau+g$, if $\eqs, \vars$ are the row resp. column dimensions of the matrix $\mat U$ of \cref{lem:U_k}, then $\vars > \eqs + 1$.
\end{definition}
Since the proposed decoder will typically return no solutions (and therefore fail) when $|\errors| > \tau$, the decoding radius tells us how many errors we should at most expect to correct.\footnote{It might succeed in rare cases in which the error locator has surprisingly low degree. Our simulation results suggest that this happens with very small probability.}
We reiterate that, at this point we have given no indications that the decoder should usually succeed up to $\taumax$ errors, but we address this by simulations in \cref{sec:num_res}.

For a given code $\mathcal C$, one may compute $\taumax$ exactly, but since the equations involve dimensions of sequences of Riemann--Roch spaces, it seems impossible to give a general precise closed form.
However, we may lower-bound these numbers using common function field tools:

\begin{lemma}
  \label{lem:div-sims}
  For $t = 1,\dots,\ell$ and $j = 0, \dots, s-1$ it holds that
  \begin{align*}
    \rrdim(V_t) &\geq \degLs + t - g + 1 \ , \\
    \rrdim(W_j) &\geq \degLs + j(\degG + \degR - n) - g + 1 \ , \\    
    \rrdim(Q_t - sD) &\geq \degLs + t(\degG + \degR) - sn - g + 1 \ ,
  \end{align*} and if $\degG \geq 2g - 2$, then 
  \[
    \rrdim(Q_t) = \degLs + t(\degG + \degR) - g + 1 \ .
  \]
\end{lemma}
\begin{proof}
  The lower bounds on $\rrdim(V_t)$, $\rrdim(W_j)$ and $\rrdim(Q_t - sD)$ are given by the Riemann-Roch theorem, while the exact value of $\rrdim(Q_t)$ is given by \cite[Thm 1.5.17]{stichtenoth_algebraic_2009} since $\degG \geq 2g - 2$.
\end{proof}

The next lemma gives bounds on the dimensions of $\mat{U}$.

\begin{lemma}
  \label{lem:U-dims}
  If $\degG \geq 2g - 2$, then $\mat{U} \in \field^{\eqs \times \vars}$, where
  \begin{align*}
    \eqs &\leq (s-1)(\degLs-g+1) + \tfrac{s(s-1)}{2}(\degG+\degR) + (\ell-s+1)sn \ , \\
    \vars &\geq (\ell+s)(\degLs-g+1)+\tfrac{\ell(\ell+1)}{2}\degG + \tfrac{s(s-1)}{2}(\degG+\degR-n) \ .
  \end{align*}
\end{lemma}
\begin{proof}
  From \cref{lem:div-sims} it follows that
  \begin{align*}
    \vars
    &= \sum_{t=1}^{\ell} \rrdim(V_t) + \sum_{j=0}^{s-1} \rrdim(W_j) \\
    &\geq \sum_{t=1}^{\ell}(\degLs + t\degG - g + 1)
      + \sum_{j=0}^{s-1}(\degLs + j(\degG + \degR) - jn - g + 1) \\
    &= \ell(\degLs - g + 1) + \tfrac{\ell(\ell+1)}{2}\degG \\
    &+ s(\degLs - g + 1) + \tfrac{(s-1)s}{2}(\degG + \degR - n) \ .
  \end{align*}
  Similarly, observing that $\mat{K}_t$ must have exactly
  \[
    \rrdim(Q_t) - \rrdim(Q_t - sD)
  \]
  rows for $t=s-1,\dots,\ell$, it holds that
  \begin{align*}
    \eqs
    &= \sum_{t=1}^{s-1} \rrdim(Q_t)
      + \sum_{t=s}^{\ell} \big(
      \rrdim(Q_t) - \rrdim(Q_t - sD)
      \big) \\
    &\leq \sum_{t=1}^{s-1} \big(
      \degLs + t(\degG + \degR) - g + 1
      \big) + \sum_{t=s}^{\ell} sn \\
    &= (s-1)(\degLs - g + 1) +  \frac{(s-1)s}{2}(\degG + \degR) + (l-s+1)sn \ .
  \end{align*}
  
\end{proof}
With the aid of \cref{lem:U-dims} we can deduce a lower bound on values of $\degLs$ for which decoding must fail.

\begin{lemma}\label{lem:lin-sys-fail}
  It holds that \eqref{eq:Uv=0} has at least two linearly independent solutions (and decoding fails) if
  \begin{equation}
    \label{eq:lin-sys-fail}
\degLs > s \tfrac{2\ell-s+1}{2(\ell+1)}n - \tfrac{\ell}{s} \degG + \tfrac{\ell}{\ell+1} +g \ .
\end{equation}
\end{lemma}
\begin{proof}
  Keeping the notation from \cref{lem:U-dims} we have that \eqref{eq:Uv=0} has at least two linearly independent solutions if
  \[
    \vars > \eqs+1 \ ,
  \]
  which is equivalent to \eqref{eq:lin-sys-fail}.
\end{proof}

\begin{corollary}
The decoding radius as defined in \cref{def:decoding_radius} is given by
\begin{align*}
\taumax(\ell,s) := \left\lfloor \tfrac{2\ell-s+1}{2(\ell+1)}n - \tfrac{\ell}{2s} \degG + \tfrac{\ell}{s(\ell+1)} \right\rfloor \ . 
\end{align*}
\end{corollary}

\begin{remark}
The decoding radius $\taumax(\ell,s)$ coincides with the one in \cite{rosenkilde2018power} in the special case of Reed--Solomon, and the one in \cite{puchinger2019improved} for one-point Hermitian codes.
\end{remark}

\begin{remark}
Although this definition of failure is convenient for theoretical analysis it means our decoder must try every sensible (sufficiently small) value of $\tau$ before declaring failure, which is computationally inefficient.
In this paper, our primary concern is not computational complexity, but we will remark that in practice, one only needs to use the largest sensible value of $\tau$: we then pick \emph{any} non-zero vector in the right kernel of $\mat U$ and check if the corresponding message $f$ is in $\rrspace(G)$.
This is how the simulations in \cref{sec:num_res} have been conducted.
\end{remark}

\subsection{Parameter Choice \& Asymptotic Behavior}

The following theorem shows that the decoding radius achieves the Johnson radius asymptotically for large parameters $\ell$ and $s$. It also implies a good practical choice of the parameters $\ell$ and $s$.

\begin{theorem}
Define the sequence $(\ell_i,s_i)_{i \in \ZZ_{>0}}$ as $\ell_i = i$ and $s_i = \left\lfloor\sqrt{\tfrac{\degG}{n}} i\right\rfloor+1$. Then, for $i \to \infty$, we have
\begin{align*}
\taumax(\ell_i,s_i) = n \left( 1- \sqrt{\tfrac{\degG}{n}} - O(\tfrac{1}{i})\right).
\end{align*}
\end{theorem}

\begin{proof}
This is a special case of \cite[Theorem~5]{puchinger2019improved}.
\end{proof}

\section{Numerical Results}\label{sec:num_res}

In this section we present Monte-Carlo simulation results of the proposed decoder for a few different AG codes in order to experimentally verify our hypothesis that Algorithm~\ref{alg:decoder} corrects up to $\taumax(\ell,s)$ errors with high probability and fails with high probability above.

To be precise, for a given code and radius $\tau$, we test the following:
\begin{itemize}
\item Draw a random message $f$ and encode it into $\c$
\item Draw an error $\e$ uniformly at random from the set of vectors in $\field^n$ of Hamming weight $\tau$
\item Decode $\r=\c+\e$ using Algorithm~\ref{alg:decoder}
\item If the algorithm does not return a decoding failure and the returned message polynomial is exactly $f$, then count it as success.
\end{itemize}
Note that the simulation might be unsuccessful also in cases in which Algorithm~\ref{alg:decoder} returns a valid close codeword, but not the transmitted one. This is called a miscorrection, and typically occurs only rarely (see, e.g., \cite{schmidt2009collaborative} for RS codes).

The above described simulation was tested on a few one- and two-point codes over two well known function fields:
\begin{itemize}
\item The \emph{Hermitian} function field is defined by the equation
  \[
    H_q : y^q + y = x^{q+1}
  \]
  over $\field_{q^2}$, i.e. the finite field of cardinality $q^2$. It has genus $\frac{1}{2}q(q-1)$  and $q^3 + 1$ rational places.
  \item The \emph{Suzuki} function field is defined by the equation
  \[
    S_q : y^q + y = x^{q_0}(x^q + x)
  \]
  over $\field_{q^4}$, where $q = 2q_0^2 > 2$. It has genus $q_0(q-1)$  and $q^2 + 1$ rational places.
  \item The function from \cite[Lemma 3.2]{garcia1996asymptotic} field is defined by
    \[
      T_q: y^q + y = \frac{x^q}{x^{q-1} + 1}
    \]
    over $\field_{q^2}$. It has genus $(q-1)^2$ and $q^3-q^2+2q$ rational places.
\end{itemize}

Table~\ref{tab:simulations} contains simulation results for various function field family, code, and decoder parameters.
All results confirm our hypothesis.

\begin{table}[ht]
    \caption{Simulation results}\label{tab:simulations}
    \resizebox{\columnwidth}{!}{
    \begin{tabular}{ccccccccllc}
      \toprule
      Curve & $|\field|$ & $\degG$ & $n$   & $k$ & $d^*$  & $\ell$ & $s$ & $\tau$ & $\text{OFR}$     & $N \geq $    \\
      \midrule
                                                                                                               \\
      $H_4$ & $4^2$    & $15$     & $64$  & $10$ & $49$ & $4$ & $2$ & $29^+$ & $0.00$                & $10^2$ \\
            &          &          &       &      &  &     &     & $30$   & $1.00$                & $10^2$ \\
      $H_4$ & $4^2$    & $10+5$   & $63$  & $10$ & $48$ & $4$ & $2$ & $28^+$ & $0.00$                & $10^2$ \\
            &          &          &       &      &  &     &     & $29$   & $1.00$                & $10^2$ \\
      $H_5$ & $5^2$    & $55$     & $125$ & $46$ & $70$ & $5$ & $2$ & $36^+$ & $0.00$                & $10^2$ \\
            &          &          &       &      &  &     &     & $37$   & $1.00$                & $10^2$ \\
      $H_5$ & $5^2$    & $30+25$  & $124$ & $46$ & $69$ & $3$ & $2$ & $35^+$ & $0.00$                & $10^2$ \\
            &          &          &       &      &  &     &     & $36$   & $0.94$                & $10^2$ \\
      $S_1$ & $2^4$    & $12$     & $24$  & $12$ & $12$ & $2$ & $2$ & $5^+$  & $0.00$                & $10^4$ \\
            &          &          &       &      &  &     &     & $6$    & $>0.99$               & $10^4$ \\
      $S_1$ & $2^4$    & $6+6$    & $23$  & $12$ & $11$ & $2$ & $2$ & $5^+$  & $0.00$                & $10^4$ \\
            &          &          &       &      &  &     &     & $6$    & $1.00$                & $10^4$ \\
      $S_1$ & $2^4$    & $4$      & $24$  & $4$  & $20$ & $6$ & $2$ & $12^+$ & $<6.98 \cdot 10^{-4}$ & $10^3$ \\
            &          &          &       &      &  &     &     & $13$   & $1.00$                & $10^3$ \\
      $S_1$ & $2^4$    & $2+2$    & $23$  & $4$  & $19$ & $6$ & $2$ & $11^+$ & $0.00$                & $10^3$ \\
            &          &          &       &      &  &     &     & $12$   & $>0.99$               & $10^3$ \\
      $T_4$ & $4^2$    & $15$    & $55$  & $7$  & $40$ & $4$ & $2$ & $23^+$ & $<1.74 \cdot 10^{-3}$                & $10^3$ \\
            &          &          &       &      &  &     &     & $24$   & $1.00$               & $10^3$ \\      
      \bottomrule
    \end{tabular}
    }
    \begin{tablenotes}
      \small
    \item Code parameters $\degG,n,k,d^*$. $\degG = a + b$ means that $G = aP + bP'$ for some rational places $P$ and $P'$. Decoder parameters $\ell,s$. Number of errors $\tau$, $^+$ means that $\tau = \tau_{\max}$. Observed failure rate $\text{OFR}$. Each simulation was repeated at least $N$ times.
    \end{tablenotes}
\end{table}

\bibliographystyle{IEEEtran}
\bibliography{bibtex}

\end{document}

%% file: defs.tex
\usepackage[utf8]{inputenc}
\usepackage[english]{babel}
\usepackage[]{amsthm} %
\usepackage[]{amssymb} %
\usepackage[]{amsmath, mathtools}
\usepackage{enumerate}
\usepackage[]{tikz}
\usepackage[capitalise]{cleveref}
\usepackage{booktabs,caption}
\usepackage[flushleft]{threeparttable}
\usepackage[linesnumbered,ruled,vlined,titlenumbered]{algorithm2e}
\usepackage[innermargin=0.545in,outermargin=0.545in,top=0.545in,bottom=.7in]{geometry}
\usepackage{graphicx}

\newtheorem{theorem}{Theorem}[section]

\newtheorem{lemma}[theorem]{Lemma}
\newtheorem{corollary}[theorem]{Corollary}
\newtheorem{remark}[theorem]{Remark}

\theoremstyle{definition}
\newtheorem{definition}[theorem]{Definition}

\newtheorem{problem}[theorem]{Problem}

\newcommand{\field}{\mathbb{F}}
\newcommand{\ffield}{F}
\renewcommand{\vec}[1]{\boldsymbol{#1}}
\newcommand{\mat}[1]{\vec{#1}}
\newcommand{\ZZ}{\ensuremath{\mathbb{Z}}}

\newcommand{\rrspace}{\mathcal{L}}
\newcommand{\rrdim}{\mathfrak{l}}
\newcommand{\ff}{F}

\newcommand{\errors}{\mathcal{E}}

\newcommand{\code}{\mathcal{C}}
\newcommand{\Ls}{\Lambda_s}
\newcommand{\degLs}{\lambda_s}
\newcommand{\degG}{\gamma}
\newcommand{\degR}{\rho}
\newcommand{\Pinf}{P_{\infty}}
\newcommand{\basis}{\vec{\beta}}
\newcommand{\taumax}{\tau_\mathsf{max}}
\newcommand{\eqs}{\epsilon}
\newcommand{\vars}{\nu}

\DeclareMathOperator{\supp}{supp}

\DeclareMathOperator{\im}{im}

\newcommand{\e}{{\vec e}}
\renewcommand{\r}{{\vec r}}
\renewcommand{\c}{{\vec c}}